\documentclass[12pt]{article}

\usepackage{amsmath}
\allowdisplaybreaks[4]
\usepackage[all]{xy}
\usepackage{amssymb}
\usepackage{amsthm}
\usepackage{hyperref}
\hypersetup{colorlinks=true,linkcolor=blue,citecolor=red}
\usepackage{amsmath}
\usepackage{amscd}
\usepackage{verbatim}
\usepackage{eurosym}
\usepackage{float}
\usepackage{color}
\usepackage{dcolumn}
\usepackage[mathscr]{eucal}
\usepackage[all]{xy}
\usepackage{hyperref}
\usepackage{mathrsfs}
\usepackage{amsmath}
\usepackage{amssymb}
\usepackage{amsfonts,ifpdf}
\usepackage{graphicx}
\usepackage{times}
\usepackage{float}
\usepackage{epstopdf}
\usepackage{cite}
\usepackage{youngtab}
\usepackage{ytableau}
\usepackage{enumerate}
\usepackage{indentfirst}
\usepackage{tikz}
\usetikzlibrary{arrows.meta,calc}
\usepackage{subcaption}

\ytableausetup
{mathmode, boxsize=0.9em}

\usepackage{listings}
\usepackage{xcolor}

\definecolor{codegray}{gray}{0.5}
\definecolor{codegreen}{rgb}{0,0.6,0}
\definecolor{codeblue}{rgb}{0.2,0.2,0.6}
\definecolor{backcolour}{rgb}{0.95,0.95,0.95}

\lstdefinestyle{mystyle}{
	backgroundcolor=\color{backcolour},
	commentstyle=\color{codegreen},
	keywordstyle=\color{codeblue},
	numberstyle=\tiny\color{codegray},
	stringstyle=\color{codeblue},
	basicstyle=\small\ttfamily,
	breakatwhitespace=false,
	breaklines=true,
	captionpos=b,
	keepspaces=true,
	numbers=left,
	numbersep=5pt,
	showspaces=false,
	showstringspaces=false,
	showtabs=false,
	tabsize=2
}

\newtheorem{theorem}{Theorem}[section]

\newtheorem{lemma}[theorem]{Lemma}
\newtheorem{prop}[theorem]{Proposition}

\newtheorem{pf}{proof}[section]
\theoremstyle{remark}
\newtheorem{exam}[theorem]{Example}

\def\A{\mathcal{A}}

\makeatletter \@addtoreset{equation}{section} \makeatother
\makeindex 

\renewcommand\ell{l}

\begin{document}
	\begin{center}
		{\Large\bf Characterizing Nice Partition of Graphical Arrangements}\\ [7pt]
	\end{center}

	\vskip 3mm
	\begin{center}
 Weikang Liang$^{1}$, Suijie Wang$^{2}$ and Chengdong Zhao$^{*3}$\\[8pt]
 $^{1,2}$School of Mathematics\\
 Hunan University\\
 Changsha 410082, Hunan, P. R. China\\[12pt]

 $^{3}$
 School of Mathematics and Statistics\\
 Central South University\\
 Changsha 410083, Hunan, P. R. China\\[15pt]

 Emails: $^{1}$kangkang@hnu.edu.cn, $^{2}$wangsuijie@hnu.edu.cn,  $^{3}$cdzhao@csu.edu.cn\\[15pt]

\end{center}

\vskip 3mm
	\begin{abstract}
The successive works of Terao as well as Stanley revealed that, for graphical arrangements,
supersolvability and the existence of nice partitions are equivalent properties, both characterized by chordal graphs.
In this paper, we further prove that every nice partition of a graphical arrangement arises precisely from a maximal modular chain in its intersection lattice.
Moreover, we establish two converses to classical results of Orlik and Terao on nice partitions.
\vskip 6pt

\\ [7pt]
\end{abstract}
	
\section{Introduction}

In the theory of hyperplane arrangements, supersolvability, freeness, and the existence of nice partitions are the most crucial properties that enable the complete factorization of characteristic polynomials. Supersolvable arrangements were initially introduced by Stanley \cite{PRS1972} subsequent to the study of supersolvable lattices. Free arrangements were introduced by Terao \cite{HT1980} to study the freeness of the derivation module. Terao \cite{HT1992} further proposed the concept of nice partitions of a hyperplane arrangement to explore the decomposition of the Orlik-Solomon algebra. It is well-known that the supersolvability of an arrangement implies both the freeness and the existence of nice partitions \cite{POHT1992}. However, the converse statements are not generally true. This means that for hyperplane arrangements, neither freeness nor the existence of nice partitions ensures supersolvability. Naturally, we may ask what conditions would make the converse statements hold. Further, it is worth investigating the conditions under which the above three properties are equivalent.

The supersolvability and nice partitions of an arrangement are determined by its combinatorics, i.e., its intersection lattice. However, Terao's conjecture \cite{HT1983}, which states that the freeness of an arrangement is characterized only by its combinatorics, remains a long-standing problem. When it comes to graphical arrangements, the combinatorics includes not only the lattice structure of arrangements, but also the combinatorial structure of graphs. Edelman and Reiner \cite{PV1994} demonstrated that a graphical arrangement is free if and only if the corresponding graph is chordal. Stanley in his well-known book \cite{PRS2007} presented that a graphical arrangement is supersolvable if and only if the graph is chordal. The following theorem, attributed to Stanley as stated in \cite[Theorem 4.1]{Bailey}, asserts that a graphical arrangement admits a nice partition precisely when the corresponding graph is chordal.
As no explicit one appears to be available in the literature to the best of our knowledge, we will include a short proof for completeness.
Therefore, for graphical arrangements, supersolvability, freeness, and the existence of nice partitions are equivalent properties, all characterized by chordal graphs.
For precise notation and terminology, we refer the reader to Section~\ref{sec-2}.

\begin{theorem}\label{chordal-nice}
Let $G$ be a simple graph.
The hyperplane arrangement $\A_{G}$ admits a nice partition if and only if $G$ is chordal.
\end{theorem}
It is worth noting that every maximal modular chain in a supersolvable arrangement induces a nice partition  \cite[Proposition~2.67, p.~50]{POHT1992}; however, the converse does not hold in general, as demonstrated by \cite[Example 3.19]{HR2016}.
In this paper, we focus on graphical arrangements and prove that the converse is true in this case; namely, all nice partitions of a graphical arrangement arise from maximal modular chains. Based on Theorem \ref{chordal-nice}, we establish the following main result.

\begin{theorem}\label{main3}
	Let $\mathcal{A}$ be a graphical arrangement in $V$ of rank $r$,
	and $L(\mathcal{A})$ denote its intersection lattice.
	Suppose that $\pi$ is a nice partition of $\mathcal{A}$.
	There is a maximal modular chain $V= X_0 < X_1 < \dots < X_{r} = T$ of $L(\A)$ such that
	the partition $\pi = \{\pi_1, \ldots, \pi_{r}\}$ is obtained by $\pi_i = \A_{X_{i}}\backslash \A_{X_{i-1}}$ for  $1\le i \le r$.
\end{theorem}

In the second part of this paper, we focus on characterizing nice partitions in a more general setting. Specifically, we prove two converses of results due to Orlik and Terao. We first establish the converse of \cite[Corollary 3.88, p.~85]{POHT1992}, thereby providing an equivalent characterization of nice partitions. The statement is as follows.

\begin{theorem}\label{main2}
	Let $\mathcal{A}$ be a hyperplane arrangement in $V$,
	and $L(\mathcal{A})$ denote its intersection lattice.
	Assume that $\pi = \{\pi_1, \ldots, \pi_{\ell}\}$ is a set partition of $\mathcal{A}$.
	Then, $\pi$ is nice if and only if for each $X \in L(\mathcal{A})$,
	\begin{equation}\label{eq:Ter}
			\chi(\mathcal{A}_X, t) = t^{n-\ell} \prod_{i=1}^{\ell} \big(t - |\pi_i \cap \mathcal{A}_X|\big)
	\end{equation}
\end{theorem}

Next we prove the converse of \cite[Proposition~2.67, p.~50]{POHT1992} and establish the following theorem, showing that if $\mathcal{A}$ has a nice partition induced by a maximal chain, then $\mathcal{A}$ must be supersolvable.

\begin{theorem}\label{main1}
	Let $\mathcal{A}$ be a hyperplane arrangement in $V$ of rank $r$,
	and $L(\mathcal{A})$ denote its intersection lattice. Assume that $\mathcal{C}\colon V = X_0 < X_1 < \dots < X_{r} = T$ is a maximal chain of $L(\mathcal{A})$, and let $\pi = \{\pi_1, \ldots, \pi_{r}\}$ be the set partition of $\mathcal{A}$ defined by $\pi_i = \mathcal{A}_{X_{i}} \setminus \mathcal{A}_{X_{i-1}}$ for $i=1, \ldots, r$.
	Then $\pi$ is a nice partition if and only if $\mathcal{C}$ is a modular chain.
\end{theorem}

The organization of this paper is as follows. In section \ref{sec-2}, we recall some concepts and notation in hyperplane arrangements and some basic results in the graph theory. In Section \ref{sec-3}, we focus on the proofs of Theorems \ref{chordal-nice} and \ref{main3}. In Section \ref{sec-4}, we prove Theorems \ref{main2} and \ref{main1}.

\section{Preliminaries}\label{sec-2}

Throughout this paper, we fix $ V = \mathbb{K}^n $ as a vector space over some field $ \mathbb{K} $.
A {\it hyperplane arrangement} $ \mathcal{A} $ is a finite family of hyperplanes in $ V $.
We refer to \cite{POHT1992, PRS2007} for standard background and terminology on hyperplane arrangements.
We assume throughout this paper that $ \mathcal{A} $ is central, that is, $ \{0\} \in \bigcap_{H \in \mathcal{A}} H $.
The \textit{intersection poset} $L(\mathcal{A})$ is the set of all intersections of hyperplanes in $\mathcal{A}$, partially ordered by reverse inclusion. The minimal element is $V$, and the maximal element $\bigcap_{H \in \mathcal{A}} H$ is denoted by $T$.
It is known that $L(\A)$ is a geometric lattice with the rank function
$r(X) = n - \dim(X)$ for $X \in L(\A)$. We set $r := r(T)$ denote the rank of $L(\A)$.
The {\it characteristic polynomial} $\chi(\mathcal{A}, t)$ of $\A$ is
\[
\chi(\mathcal{A}, t) = \sum_{X \in {L}(\mathcal{A})} \mu(X) t^{\dim(X)},
\]
where $\mu$ is known as the {\it M\"{o}bius function}, defined recursively by $\mu(V) = 1$ and
$\mu(X) = - \sum_{Y < X}\mu(Y)$ for all $X \neq V$.
Clearly, $\chi(\A, 1) = 0$.
Recall that for any $X, Y \in L(\A)$, their {\it join} is  $X \vee Y = X \cap Y$, and their {\it meet} is
$X \wedge Y = \bigcap_{ Z \in L(\A);  X,Y  \subseteq Z}Z$.
Let $\mathcal{A}_1$ and $\mathcal{A}_2$ be hyperplane arrangements defined in vector spaces $V_1$ and $V_2$, respectively.
The {\it product arrangement} $\mathcal{A}_1 \times \mathcal{A}_2$ in $V_1 \oplus V_2$ is defined by
\[
\mathcal{A}_1 \times \mathcal{A}_2 = \{H_1 \oplus V_2 \mid H_1 \in \mathcal{A}_1\} \cup \{V_1 \oplus H_2 \mid H_2 \in \mathcal{A}_2\}.
\]
With the partial order defined by $(X_1, X_2) \leq (Y_1, Y_2)$ if $X_1 \leq Y_1$ and $X_2 \leq Y_2$, the poset $L(\mathcal{A}_1) \times L(\mathcal{A}_2)$ forms a lattice.
There is a natural isomorphism of lattices
\begin{align}\label{iso}
	\sigma \colon L(\mathcal{A}_1) \times L(\mathcal{A}_2) \longrightarrow L(\mathcal{A}_1 \times \mathcal{A}_2)
\end{align}
given by the map $\sigma(X_1, X_2) = X_1 \oplus X_2$; see \cite[Proposition 2.14]{POHT1992}.

We call an element $X \in L(\A)$ {\it modular} if for any $Y \in L(\A)$, the rank function satisfies
	$
	r(X)+r(Y) = r(X \vee Y) + r(X \wedge Y),
	$
	or equivalently, $X + Y \in L(\A)$ holds.
A maximal chain
$
V = X_0 < X_1 < \dots < X_{r} = T
$
of $L(\mathcal{A})$ consisting of modular elements $X_i$ is called a {\it maximal modular chain}.
The hyperplane arrangement $\mathcal{A}$, or its intersection lattice $L(\mathcal{A})$, is called {\it supersolvable} if there exists a maximal modular chain in $L(\mathcal{A})$.

Let $\pi = \{\pi_1, \ldots, \pi_{\ell}\}$ be a set partition of $\mathcal{A}$. That is, $\pi_1 \cup \cdots \cup \pi_{\ell} = \mathcal{A}$ and $\pi_i \cap \pi_j = \emptyset$ for $i \neq j$.
Each $\pi_i$ is called a part of $\pi$.
A $p$-tuple $S = (H_1, \ldots, H_p)$ of $\mathcal{A}$ is called a \textit{$p$-section} of $\pi$ if any two distinct hyperplanes in $S$ belong to different parts of $\pi$.
Additionally, we define the $0$-section to be $ S = () $.
The set of all $p$-sections of $\pi$ for all $p \geq 0$ is denoted by $\mathcal{S}(\pi)$.
We call the $p$-tuple $S = (H_1, \ldots, H_p)$ of $\mathcal{A}$ {\it independent} if $r\left(\bigcap_{i=1}^p H_i\right) = p $ and call the partition $\pi$ of $\mathcal{A}$ {\it independent} if all $S \in \mathcal{S}(\pi)$ are independent.
The partition $ \pi $ of $ \mathcal{A} $ is {\it nice} if $ \pi $ is independent and, for any $ X \in L(\mathcal{A})\setminus\{V\} $, there exists some $i \in \{1, \ldots, \ell\}$ such that $ |\pi_i \cap \mathcal{A}_X| = 1 $, where $ \mathcal{A}_X = \{ H \in \mathcal{A} \mid X \subseteq H \} $ denotes the {\it localization} of $ \mathcal{A} $ at $ X $.
Assume that $ \pi = \{\pi_1, \ldots, \pi_{\ell}\} $ is a nice partition of $ \mathcal{A}$.
For each $ X \in L(\mathcal{A}) $, let
$
\pi_X = \{\pi_i \cap \mathcal{A}_X \mid \pi_i \cap \mathcal{A}_X \neq \emptyset, \, i = 1, \ldots, \ell\}.
$
Then $\pi_X$ is a nice partition of  $ \mathcal{A}_X $; see \cite[Corollary 2.11]{HT1992}  for more details.
We say that $\A$ is \emph{factored} provided $\A$ admits a nice partition. Note that the original definition of a nice partition in \cite{HT1992} is an ordered set partition of $\mathcal{A}$.
We treat it as an unordered partition for simplicity, without affecting the essence.

Next, we present the definition of a chordal graph. Let $ G = ([n], E) $ be a simple graph,
where $[n] = \{1, \ldots, n\}$. Each edge $\{i, j\} \in E$ is denoted by $ij \in E$ for simplicity.
A vertex $u$ in a graph $G$ is  \textit{simplicial} if the subgraph induced by its neighborhood forms a clique, that is, every pair of distinct vertices in the neighborhood of $u$ is connected by an edge.
A {\it simplicial elimination ordering} is an ordering $\sigma_1 \prec \sigma_2\prec \cdots\prec \sigma_n$ of the vertices $[n]$ such that each vertex $\sigma_i$ is simplicial in the subgraph induced by the vertices $\{\sigma_i, \sigma_{i+1}, \ldots, \sigma_n\}$. Recall that a {\it chordal graph} is a graph in which every cycle of four or more vertices has a chord, which is an edge connecting two non-adjacent vertices in the cycle.
The following is a fundamental fact about chordal graphs; see \cite[Theorem~5.3.17]{DB2002} for details.

\begin{prop}\label{graphthy4.1}
	A graph $G$ is chordal if and only if $G$ has a simplicial elimination ordering.
\end{prop}

A graph is {\it doubly connected} if it is connected and remains connected after the removal of any single vertex.
A {\it block} of a connected graph is a maximal subgraph that is doubly connected.
By definition, two distinct blocks intersect in at most one vertex.
Hence each connected graph is formed by its blocks linked together through some shared vertices. Let $G$ be a simple graph.
The hyperplane arrangement
$ \mathcal{A}_G = \{
H_{ij}\colon x_i - x_j = 0\mid  ij \in E\}
$ in $V$
is called the {\it graphical arrangement induced by $ G $}. If $G_1, G_2, \ldots, G_k$ are all the blocks of $G$, then
\begin{align}\label{product-decomposition}
	L(\mathcal{A}_G) \cong L(\mathcal{A}_{G_1}) \times L(\mathcal{A}_{G_2})  \times  \cdots \times L(\mathcal{A}_{G_k}).
\end{align}
The subsequent result is based on Stanley's characterization of supersolvable graphical arrangements, as presented in \cite[Corollary 4.10]{PRS2007}.

\begin{prop}\label{stanley4.10}
	Let $G$ be a simple graph. The arrangement $\A_G$ is supersolvable if and only if $G$ has a simplicial elimination ordering.
\end{prop}

\section{Nice partitions in graphical arrangements}\label{sec-3}

\subsection{Proof of Theorem \ref{chordal-nice}}	
We shall reduce the proof of Theorem \ref{chordal-nice} to the doubly connected components of $ G $.
The following lemma shows that the product-decomposition process \eqref{product-decomposition} preserves the nice partition of $\mathcal{A}$.
One can refer to \cite[Proposition 3.29]{HR2016} for this result in a different context.
\begin{lemma}\label{lem:1}
Let $G$ be a simple connected graph with blocks $G_1, G_2, \ldots, G_k$.
Then $\mathcal{A}_G$ is factored if and only if $\mathcal{A}_{G_i}$ is factored for all $i \in \{1, \ldots, k\}$.
\end{lemma}

\begin{proof}
For each $1 \le i \le k$,
set
\[X_i = \bigcap_{H \in \mathcal{A}_{G_i}}H \in L(\mathcal{A}_G).\]
It is clear that $\mathcal{A}_{X_i} = \mathcal{A}_{G_i}$.
As a result,
$\pi_{X_i}$ is a nice partition of $\mathcal{A}_{G_i}$ for any $i \in \{1, \ldots, k\}$.
Conversely, let $\pi^i$ be a nice partition of $\mathcal{A}_{G_i}$ for $i=1, \ldots, k$.
It can be readily verified that $\pi = \pi^1 \cup \cdots \cup \pi^k$ forms a nice partition of $\mathcal{A}_G$.
\end{proof}

To illustrate the lemma more clearly, we provide a concrete example.

\begin{exam}\label{ex}
Let $G$ be as shown in the Figure \ref{fig:example}, with $G_1$ and $G_2$ being the blocks corresponding to the vertex sets $\{1, 2, 3, 4\}$ and $\{4, 5, 6\}$, respectively.
It is easy to verify that $\pi^1 = \{\{H_{24}\}, \{H_{12}, H_{14}\}, \{H_{23}, H_{34}\}\}$ and $\pi^2 = \{\{H_{56}\}, \{H_{45}, H_{46}\}\}$ are nice partitions of $\A_{G_1}$ and $\A_{G_2}$, respectively.
Then, the union $\pi^1 \cup \pi^2$ forms a nice partition of $\A_G$.
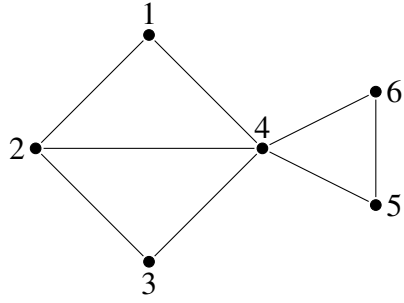
\begin{figure}[h!]
    \centering
    \begin{tikzpicture}[scale=1.5, every node/.style={inner sep=0, outer sep=1, minimum size=4pt}]
        \node[draw, fill=black, circle, label=above:{1}] (1) at (0, 2) {};
        \node[draw, fill=black, circle, label=left:{2}] (2) at (-1, 1) {};
        \node[draw, fill=black, circle, label=below:{3}] (3) at (0, 0) {};
        \node[draw, fill=black, circle, label=above:{4}] (4) at (1, 1) {};
        \node[draw, fill=black, circle, label=right:{5}] (5) at (2, 0.5) {};
        \node[draw, fill=black, circle, label=right:{6}] (6) at (2, 1.5) {};

        \draw (1) -- (2);
        \draw (1) -- (4);
        \draw (2) -- (3);
        \draw (3) -- (4);
        \draw (2) -- (4);
        \draw (4) -- (5);
        \draw (4) -- (6);
        \draw (5) -- (6);
    \end{tikzpicture}
    \caption{The graph $G$ in Example \ref{ex}, which consists of two blocks}
    \label{fig:example}
\end{figure}
\end{exam}

\begin{proof}[Proof of Theorem \ref{chordal-nice}]
Combining Propositions \ref{graphthy4.1} and \ref{stanley4.10},
we see that a graphical arrangement is supersolvable if and only if the corresponding graph is chordal.
Further,
by applying the well-known result \cite[Proposition 2.67, p.~50]{POHT1992}, which shows that every supersolvable arrangement admits a nice partition,
we conclude that the arrangement induced by a chordal graph is factored, establishing the sufficiency.

Therefore, it remains to show that the graph $G$ must be chordal whenever $\A_G$ is factored.
According to Lemma \ref{lem:1}, we can assume that $G$ is doubly connected.
Let $\pi = \{\pi_1, \ldots, \pi_{\ell}\}$ be a nice partition of $\A_G$.
The case when $G$ is acyclic is trivial.
Suppose that $G$ contains a chordless cycle $C = (e_1, \ldots, e_k)$ with $k \geq 4$, where each $e_i$ is an edge of $G$. For $e_i, e_j \in C$, set $X = H_{e_i} \cap H_{e_j}$.
Then we have $(\A_G)_X = \{H_{e_i}, H_{e_j}\}$ due to the assumption that $C$ has no chord.
Since $\pi$ is a nice partition of $\A_G$,
we have $|(\A_G)_X \cap \pi_t| = 1$ for some $t \in \{1, \ldots, \ell \}$.
Therefore, $H_{e_i}$ and $H_{e_j}$ must be in different parts of $\pi$.
This means that $H_{e_1}, H_{e_2}, \ldots, H_{e_k}$ all lie in different parts of $\pi$.
Hence, $S = (H_{e_1}, \ldots, H_{e_k})$ forms a $k$-section of $\pi$.
Since $\pi$ is independent, we know $S$ is independent, which contradicts the assumption that $C$ is a cycle.
Therefore, $G$ is a chordal graph.
\end{proof}

\subsection{Proof of Theorem \ref{main3}}	
Our next task is to prove Theorem \ref{main3},
which asserts that every nice partition of a factored graphical arrangement $\A$ is induced by a maximal modular chain of $L(\A)$. In other words, we can construct all the nice partitions of $\A$ through maximal chains of $L(\A)$.
To this end, we will lay the groundwork with some lemmas. Let $\pi$ be a nice partition of $\A$. It follows from \cite[Corollary 3.90, p.~85]{POHT1992} that $r(X) = r(L(\A_X)) = |\pi_X|$ for any $X \in L(\A)$. For instance, when $ G $ is connected, the rank of $ L(\mathcal{A}_G) $ is $n - 1$. Hence, the number of parts in its nice partition equals $ n - 1 $, provided that $ \mathcal{A}_G $ is factored.
	
\begin{lemma}\label{lem1}
Let $ \pi $ be a nice partition of $ \A_G $. Then for any triangle $ T $ in $ G $, let $ X = \bigcap_{H \in \A_T} H $. The partition $ \pi_X $ consists of two parts, one of size $1$ and the other of size $2$.
	\end{lemma}
	
	\begin{proof}
Note that the partition $\pi_X$ is nice since $\pi$ is a nice partition. Therefore, $ r(X) = 2 $ implies that $ \pi_X $ contains exactly two non-empty parts, as required.
	\end{proof}
	
	\begin{lemma}\label{lem2}
Let $ \pi $ be a nice partition of $ \A_G $.
If $ H_{ij} $ and $ H_{jk} $ are in the same part of $ \pi $,
then $ ik $ must be an edge of $ G $, and $ H_{ik} $ must belong to a distinct part of $ \pi $.
	\end{lemma}
	\begin{proof}
Write $ X = H_{ij} \cap H_{jk} $.
Suppose $ ik \notin E(G) $, which implies that $ \A_X = \{H_{ij}, H_{jk}\} $.
Consequently, $ \pi_X $ has only one part with size $2$,
which does not equal the $ r(X) $, leading to a contradiction. Thus, $ ik \in E(G) $, and $ H_{ik} $ lies in another part of $ \pi $ by Lemma \ref{lem1}.
	\end{proof}

\begin{lemma}\label{mainprop1}
Let $ G $ be a doubly connected chordal graph.
Suppose that $ \pi = \{\pi_1, \ldots, \pi_{n-1}\} $ is a nice partition of $ \mathcal{A}_G $. Then
\begin{enumerate}
    \item[{\rm (1)}]
    For each $ i \in \{1, \ldots, n-1\} $, the edges corresponding to the hyperplanes in $ \pi_i $ are incident to a common vertex, denoted $ v_i $.
    \item[{\rm (2)}] Under the notation of {\rm (1)}, for $ i \neq j $, we have $ v_i \neq v_j $.
\end{enumerate}
\end{lemma}	

\begin{proof}
Due to the one-to-one correspondence between edges and hyperplanes,
for simplicity, edges and hyperplanes will not be explicitly distinguished unless necessary.

To prove (1), note the following simple facts for any graphical arrangement $\A_G$:
For any $X \in L(\A_G)$ with $\text{rank}(X)=2$, we have $|\A_X| \leq 3$, and besides,
\begin{enumerate}
    \item[(a)] If $|\A_X| = 3$, then the hyperplanes in $\A_X$ form a triangle in $G$.
    \item[(b)] If $|\A_X| = 2$, then either the hyperplanes in $\A_X$ do not share a common vertex in $G$, or there is no edge connecting the remaining two vertices.
\end{enumerate}
For any $H, K \in \pi_i$, since $\text{rank}(H \cap K) = 2$, $\pi_{H \cap K}$ consists of two parts.
As $H$ and $K$ already belong to the same part of $\pi_{H \cap K}$, the other part must be non-empty, implying $|\mathcal{A}_{H \cap K}| = 3$.
By (a) of the above facts,
$H$ and $K$ must form two sides of a triangle in $G$, sharing a common vertex.
In summary, any two hyperplanes in $\pi_i$ share a common vertex.
However, Lemma \ref{lem2} ensures that the hyperplanes in $\pi_i$ cannot form a triangle.
Thus, all hyperplanes in $\pi_i$ form a star-shaped graph, where every hyperplane is connected to a single common vertex $v_i$.

To prove (2), suppose for contradiction that $v_1 = v_2 = v$.
Let
\[
A = \{u \in [n] \mid {vu} \in \pi_1\} \quad \text{and} \quad B = \{w \in  [n] \mid {vw} \in \pi_2\}.
\]
Since $G$ is doubly connected,
we can choose a shortest path $P = [u_0,u_1,\ldots ,u_k]$ from $A$ to $B$ that does not pass through $v$.
By the minimality of the length of $P$,
we have $u_0 \in A$, $u_k \in B$, and $u_1, \ldots, u_{k-1} \in  [n] \setminus (A \cup B)$.
See Figure \ref{Fig:com} for an illustration.
If $ k = 1 $,
then $ {vu_0} $, $ {vu_1} $, and $ {u_0u_1} $ form a triangle.
By Lemma \ref{lem1}, $ {u_0u_1} $ must belong to either $ \pi_1 $ or $ \pi_2 $.
Regardless, it should be connected to vertex $ v $ by (1) of this Lemma, which leads to a contradiction.
For $ k \geq 2 $,
consider the cycle $ [v, u_0, u_1, \cdots u_k, v] $ of length $ k+2 $.
Given that $G$ is chordal and by the minimality of the path $P$, it follows that each $vu_i$ is an edge of $G$ for all $i \in \{1, \ldots, k-1\}$.
In the triangle $ [v,u_0,u_1,v] $, since $u_0u_1, vu_1 \not\in \pi_1 $, Lemma \ref{lem1} implies that $ u_0u_1 $ and $ vu_1 $ are in the same part of $\pi$.
On the one hand, $ u_0u_1 $ and $ u_1u_2 $ cannot lie in the same part of $\pi$; otherwise, by Lemma \ref{lem2}, $ u_0u_2 $ would be an edge in $ G $, contradicting the minimality of the path $P$.
On the other hand, by (1) of this Lemma, $u_0u_1$ and $vu_2$ cannot lie in the same part of $\pi$.
Hence, in the triangle $ [v,u_1,u_2,v] $, by Lemma \ref{lem1}, we conclude that $ u_1u_2 $ and $ vu_2 $ are in the same part of $\pi$. Proceeding in this manner to the end, we find that $u_{k-1}u_k$ and $vu_k$ are in the same part of $\pi$, which means $u_{k-1}u_k \in \pi_2 $, leading to a contradiction.
\tikzset{every picture/.style={line width=0.75pt}} 
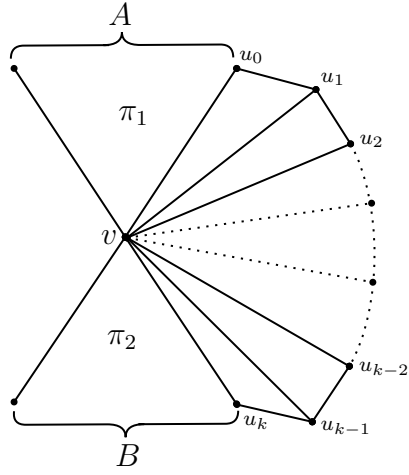
\begin{figure}[H]
\centering
			\begin{tikzpicture}[x=0.75pt,y=0.75pt,yscale=-0.93,xscale=0.93]

				\draw    (290,59) -- (350,150) ;
				\draw [shift={(350,150)}, rotate = 56.7] [color={rgb, 255:red, 0; green, 0; blue, 0 }  ][fill={rgb, 255:red, 0; green, 0; blue, 0 }  ][line width=0.75]      (0, 0) circle [x radius= 1.34, y radius= 1.34]   ;
				\draw [shift={(290,59)}, rotate = 56.7] [color={rgb, 255:red, 0; green, 0; blue, 0 }  ][fill={rgb, 255:red, 0; green, 0; blue, 0 }  ][line width=0.75]      (0, 0) circle [x radius= 1.34, y radius= 1.34]   ;
				\draw    (410,59) -- (350,150) ;
				\draw [shift={(350,150)}, rotate = 123.01] [color={rgb, 255:red, 0; green, 0; blue, 0 }  ][fill={rgb, 255:red, 0; green, 0; blue, 0 }  ][line width=0.75]      (0, 0) circle [x radius= 1.34, y radius= 1.34]   ;
				\draw [shift={(410,59)}, rotate = 123.01] [color={rgb, 255:red, 0; green, 0; blue, 0 }  ][fill={rgb, 255:red, 0; green, 0; blue, 0 }  ][line width=0.75]      (0, 0) circle [x radius= 1.34, y radius= 1.34]   ;
				\draw    (350,150) -- (410,240) ;
				\draw [shift={(410,240)}, rotate = 56.6] [color={rgb, 255:red, 0; green, 0; blue, 0 }  ][fill={rgb, 255:red, 0; green, 0; blue, 0 }  ][line width=0.75]      (0, 0) circle [x radius= 1.34, y radius= 1.34]   ;
				\draw [shift={(350,150)}, rotate = 56.6] [color={rgb, 255:red, 0; green, 0; blue, 0 }  ][fill={rgb, 255:red, 0; green, 0; blue, 0 }  ][line width=0.75]      (0, 0) circle [x radius= 1.34, y radius= 1.34]   ;
				\draw    (350,150) -- (290,239) ;
				\draw [shift={(290,239)}, rotate = 124.67] [color={rgb, 255:red, 0; green, 0; blue, 0 }  ][fill={rgb, 255:red, 0; green, 0; blue, 0 }  ][line width=0.75]      (0, 0) circle [x radius= 1.34, y radius= 1.34]   ;
				\draw [shift={(350,150)}, rotate = 124.67] [color={rgb, 255:red, 0; green, 0; blue, 0 }  ][fill={rgb, 255:red, 0; green, 0; blue, 0 }  ][line width=0.75]      (0, 0) circle [x radius= 1.34, y radius= 1.34]   ;
				\draw   (290,244) .. controls (289.36,249) and (291.7,251.32) .. (296.37,251.3) -- (340.17,251.08) .. controls (346.84,251.05) and (350.18,253.37) .. (350.2,258.04) .. controls (350.18,253.37) and (353.5,251.02) .. (360.17,250.99)(357.17,251) -- (403.97,250.77) .. controls (408.64,250.75) and (410.96,248.41) .. (410.93,243.74) ;
				\draw   (408.27,53.07) .. controls (408.24,48.4) and (405.9,46.08) .. (401.23,46.11) -- (358.23,46.35) .. controls (351.56,46.38) and (348.22,44.07) .. (348.19,39.4) .. controls (348.22,44.07) and (344.9,46.42) .. (338.23,46.46)(341.23,46.44) -- (295.23,46.7) .. controls (290.56,46.73) and (288.24,49.07) .. (288.27,53.74) ;
				\draw    (350.93,150.4) -- (452.93,70.4) ;
				\draw    (410.27,59.07) -- (452.93,70.4) ;
				\draw [shift={(452.93,70.4)}, rotate = 14.88] [color={rgb, 255:red, 0; green, 0; blue, 0 }  ][fill={rgb, 255:red, 0; green, 0; blue, 0 }  ][line width=0.75]      (0, 0) circle [x radius= 1.34, y radius= 1.34]   ;
				\draw [shift={(410.27,59.07)}, rotate = 14.88] [color={rgb, 255:red, 0; green, 0; blue, 0 }  ][fill={rgb, 255:red, 0; green, 0; blue, 0 }  ][line width=0.75]      (0, 0) circle [x radius= 1.34, y radius= 1.34]   ;
				\draw    (350.93,150.4) -- (471.6,99.73) ;
				\draw    (452.93,70.4) -- (471.6,99.73) ;
				\draw [shift={(471.6,99.73)}, rotate = 57.53] [color={rgb, 255:red, 0; green, 0; blue, 0 }  ][fill={rgb, 255:red, 0; green, 0; blue, 0 }  ][line width=0.75]      (0, 0) circle [x radius= 1.34, y radius= 1.34]   ;
				\draw [shift={(452.93,70.4)}, rotate = 57.53] [color={rgb, 255:red, 0; green, 0; blue, 0 }  ][fill={rgb, 255:red, 0; green, 0; blue, 0 }  ][line width=0.75]      (0, 0) circle [x radius= 1.34, y radius= 1.34]   ;
				\draw  [dash pattern={on 0.84pt off 2.51pt}]  (471.6,99.73) .. controls (488.93,131.07) and (488.93,189.07) .. (470.93,219.73) ;
				\draw [shift={(470.93,219.73)}, rotate = 120.41] [color={rgb, 255:red, 0; green, 0; blue, 0 }  ][fill={rgb, 255:red, 0; green, 0; blue, 0 }  ][line width=0.75]      (0, 0) circle [x radius= 1.34, y radius= 1.34]   ;
				\draw [shift={(471.6,99.73)}, rotate = 61.05] [color={rgb, 255:red, 0; green, 0; blue, 0 }  ][fill={rgb, 255:red, 0; green, 0; blue, 0 }  ][line width=0.75]      (0, 0) circle [x radius= 1.34, y radius= 1.34]   ;
				\draw    (350.93,150.4) -- (470.93,219.73) ;
				\draw [shift={(470.93,219.73)}, rotate = 30.02] [color={rgb, 255:red, 0; green, 0; blue, 0 }  ][fill={rgb, 255:red, 0; green, 0; blue, 0 }  ][line width=0.75]      (0, 0) circle [x radius= 1.34, y radius= 1.34]   ;
				\draw [shift={(350.93,150.4)}, rotate = 30.02] [color={rgb, 255:red, 0; green, 0; blue, 0 }  ][fill={rgb, 255:red, 0; green, 0; blue, 0 }  ][line width=0.75]      (0, 0) circle [x radius= 1.34, y radius= 1.34]   ;
				\draw    (470.93,219.73) -- (450.93,249.73) ;
				\draw [shift={(450.93,249.73)}, rotate = 123.69] [color={rgb, 255:red, 0; green, 0; blue, 0 }  ][fill={rgb, 255:red, 0; green, 0; blue, 0 }  ][line width=0.75]      (0, 0) circle [x radius= 1.34, y radius= 1.34]   ;
				\draw [shift={(470.93,219.73)}, rotate = 123.69] [color={rgb, 255:red, 0; green, 0; blue, 0 }  ][fill={rgb, 255:red, 0; green, 0; blue, 0 }  ][line width=0.75]      (0, 0) circle [x radius= 1.34, y radius= 1.34]   ;
				\draw    (410.27,240.4) -- (450.93,249.73) ;
				\draw [shift={(450.93,249.73)}, rotate = 12.93] [color={rgb, 255:red, 0; green, 0; blue, 0 }  ][fill={rgb, 255:red, 0; green, 0; blue, 0 }  ][line width=0.75]      (0, 0) circle [x radius= 1.34, y radius= 1.34]   ;
				\draw [shift={(410.27,240.4)}, rotate = 12.93] [color={rgb, 255:red, 0; green, 0; blue, 0 }  ][fill={rgb, 255:red, 0; green, 0; blue, 0 }  ][line width=0.75]      (0, 0) circle [x radius= 1.34, y radius= 1.34]   ;
				\draw  [dash pattern={on 0.84pt off 2.51pt}]  (350.93,150.4) -- (482.93,131.73) ;
				\draw [shift={(482.93,131.73)}, rotate = 351.95] [color={rgb, 255:red, 0; green, 0; blue, 0 }  ][fill={rgb, 255:red, 0; green, 0; blue, 0 }  ][line width=0.75]      (0, 0) circle [x radius= 1.34, y radius= 1.34]   ;
				\draw [shift={(350.93,150.4)}, rotate = 351.95] [color={rgb, 255:red, 0; green, 0; blue, 0 }  ][fill={rgb, 255:red, 0; green, 0; blue, 0 }  ][line width=0.75]      (0, 0) circle [x radius= 1.34, y radius= 1.34]   ;
				\draw  [dash pattern={on 0.84pt off 2.51pt}]  (350.93,150.4) -- (483.6,174.4) ;
				\draw [shift={(483.6,174.4)}, rotate = 10.25] [color={rgb, 255:red, 0; green, 0; blue, 0 }  ][fill={rgb, 255:red, 0; green, 0; blue, 0 }  ][line width=0.75]      (0, 0) circle [x radius= 1.34, y radius= 1.34]   ;
				\draw [shift={(350.93,150.4)}, rotate = 10.25] [color={rgb, 255:red, 0; green, 0; blue, 0 }  ][fill={rgb, 255:red, 0; green, 0; blue, 0 }  ][line width=0.75]      (0, 0) circle [x radius= 1.34, y radius= 1.34]   ;
				\draw    (350.93,150.4) -- (450.93,249.73) ;
				\draw [shift={(450.93,249.73)}, rotate = 44.81] [color={rgb, 255:red, 0; green, 0; blue, 0 }  ][fill={rgb, 255:red, 0; green, 0; blue, 0 }  ][line width=0.75]      (0, 0) circle [x radius= 1.34, y radius= 1.34]   ;
				\draw [shift={(350.93,150.4)}, rotate = 44.81] [color={rgb, 255:red, 0; green, 0; blue, 0 }  ][fill={rgb, 255:red, 0; green, 0; blue, 0 }  ][line width=0.75]      (0, 0) circle [x radius= 1.34, y radius= 1.34]   ;
				
				\draw (339.33,22.33) node [anchor=north west][inner sep=0.75pt]   [align=left] {$A$};
				\draw (343,260) node [anchor=north west][inner sep=0.75pt]   [align=left] {$B$};
				\draw (410,48) node [anchor=north west][inner sep=0.75pt]  [font=\scriptsize] [align=left] {$u_0$};
				\draw (454,60) node [anchor=north west][inner sep=0.75pt]  [font=\scriptsize] [align=left] {$u_1$};
				\draw (475,93) node [anchor=north west][inner sep=0.75pt]  [font=\scriptsize] [align=left] {$u_2$};
				\draw (412.27,245) node [anchor=north west][inner sep=0.75pt]  [font=\scriptsize] [align=left] {$u_k$};
				\draw (454,247.67) node [anchor=north west][inner sep=0.75pt]  [font=\scriptsize] [align=left] {$u_{k-1}$};
				\draw (474.67,216) node [anchor=north west][inner sep=0.75pt]  [font=\scriptsize] [align=left] {$u_{k-2}$};
				\draw (344.67,77.67) node [anchor=north west][inner sep=0.75pt]   [align=left] {$\pi_1$};
				\draw (338.67,198.67) node [anchor=north west][inner sep=0.75pt]   [align=left] {$\pi_2$};
				\draw (335.33,145) node [anchor=north west][inner sep=0.75pt]   [align=left] {$v$};			
			\end{tikzpicture}
			\caption{The illustration of Lemma \ref{mainprop1}}\label{Fig:com}
\end{figure}
		\end{proof}

\begin{lemma}\label{lem-3}
Let $ G $ be a doubly connected chordal graph. Then each nice partition of $\A_G$ has exactly one part of size $1$.
\end{lemma}	
\begin{proof}
Let $\pi = \{\pi_1, \pi_2, \ldots, \pi_{n-1}\}$ be a nice partition of $\A_G$,
and let $v_i$ be the common vertex for the edges in $\pi_i$ as used in Lemma \ref{mainprop1}.
The existence of a part of size $1$ in $\pi$ follows directly from the definition of a nice partition.
For uniqueness, assume to the contrary that $|\pi_1| = |\pi_2| = 1$.
Let $\pi_1 = \{H_{v_1u_1}\}$ and $\pi_2 = \{H_{v_2u_2}\}$.
Then, by applying the same reasoning in the proof of Lemma \ref{mainprop1}, we have $\{u_1, u_2\} \cap \{v_1, v_2, \ldots, v_{n-1}\} = \emptyset$,
which implies that $G$ contains at least $ n + 1$ vertices, a contradiction.
\end{proof}
	
	For a connected graph $G$ with doubly connected blocks $G_1, G_2, \ldots, G_k$, the following lemma allows us to construct a modular chain in $L(\mathcal{A}_G)$ from those of $L(\mathcal{A}_{G_1}), \ldots, L(\mathcal{A}_{G_k})$. We provide a proof for completeness; see \cite[Exercise (18), page~447]{PRS2007} and \cite[Proposition 2.5]{HR2014} for slightly different statements of this result.

\begin{lemma}\label{lem-4}
Let $\mathcal{A}_1$ and $\mathcal{A}_2$ be hyperplane arrangements in vector spaces $V_1$ and $V_2$, respectively.
If $X$ is a modular element in $L(\mathcal{A}_1)$ and $Y$ is a modular element in $L(\mathcal{A}_2)$,
then $X \oplus Y$ is modular in $L(\mathcal{A}_1 \times \mathcal{A}_2)$ under the isomorphism $\pi$ as defined in \eqref{iso}.
\end{lemma}
	
\begin{proof}
A well-known result, \cite[Corollary 2.26]{POHT1992}, states that an element $ Z \in L(\A) $ is modular if and only if $ Z + W \in L(\A) $ for all $ W \in L(\A) $.
Based on this, within the context of this proposition, take any $ W \in L(\A_1 \times \A_2) $. Suppose $ W = X' \oplus Y' $, where $ X' \in L(\A_1) $ and $ Y' \in L(\A_2) $. Thus, we have
\[
(X \oplus Y) + W = (X \oplus Y) + (X' \oplus Y') = (X + X') \oplus (Y + Y').
\]
Since $ X \in L(\A_1) $ and $ Y \in L(\A_2) $ are modular, we have $ X + X' \in L(\A_1) $ and $ Y + Y' \in L(\A_2) $. Therefore, $(X \oplus Y) + W \in L(\A_1 \times \A_2)$, which completes the proof.
\end{proof}	

Ultimately, we arrive at the proof of Theorem \ref{main3}.

	\begin{proof}[Proof of Theorem \ref{main3}]
Assume that $\A = \A_G$. Then the graph $G$ is chordal by Theorem \ref{chordal-nice}.
We first address the case when $G$ is doubly connected.
By Lemma \ref{lem-3}, we may assume without loss of generality that $\pi_1$ is the only part of size $1$.
Using the notation of Lemma \ref{mainprop1}, let $v_i$ be the common vertex of all edges in $\pi_i$;
for the single edge in $\pi_1$, we arbitrarily fix one of its two possible vertices.
We then orient $G$ to get a directed graph $D(G)$ as follows: for any $i \in \{1, \ldots, n-1\}$, if $H_{v_iu} \in \pi_i$, then the edge $v_iu$ is directed from $v_i$ to $u$.
We claim that there are no directed cycles in $D(G)$; otherwise, consider a directed cycle $v_{i_1} \rightarrow v_{i_2} \rightarrow \cdots \rightarrow v_{i_k} \rightarrow v_{i_1}$ in $D(G)$.
Clearly, this comes from the undirected cycle $[v_{i_1}, v_{i_2}, \ldots, v_{i_k}, v_{i_1}]$ of $G$. Then $H_{v_{i_j}v_{i_{j+1}}} \in \pi_{i_j}$ for any $j = 1, \ldots k-1$ and $H_{v_{i_k}v_{i_{1}}} \in \pi_{i_k}$ by Lemma \ref{mainprop1}. Therefore, the tuple $(H_{v_{i_1}v_{i_2}}, H_{v_{i_2}v_{i_3}}, \ldots, H_{v_{i_k}v_{i_1}})$ is a $k$-section of $\pi$, which must be independent since $\pi$ is a nice partition.
This contradicts the fact that $[v_{i_1}, v_{i_2}, \ldots, v_{i_k}, v_{i_1}]$ is a cycle in $G$. Thus, our claim holds. 
Consequently, there exists an ordering $\sigma\colon \sigma_1\prec \sigma_2\prec \ldots\prec \sigma_{n}$ of the vertex set $[n]$ such that for every edge $\sigma_i \sigma_j$ in $D(G)$ directed from $\sigma_i$ to $\sigma_j$, we have $\sigma_i \prec \sigma_j$; see \cite[Exercise 1.4.14]{DB2002}.
Together with Lemma \ref{lem2}, $\sigma$ is a simplicial elimination ordering.
Combining Theorem \ref{chordal-nice} with Stanley's Proposition \ref{stanley4.10}, we can construct a modular chain as follows.
For each $i \in \{1, \ldots, n-1\}$, choose any hyperplane $H^i$ corresponding to an outgoing edge from the vertex $\sigma_{i}$. Then, let
\[
X_0 = V \quad \text{and} \quad X_i = X_{i-1} \cap H^{n-i} \quad \text{for } i \in \{1, \ldots, n-1\}.
\]
This yields a maximal modular chain $\mathcal{C}\colon X_0 < X_1 < X_2< \cdots  < X_{n-1}$ of $L(\A_G)$, which induces the nice partition $\pi$.
For a concrete example, refer to Example \ref{ex:mc}.

Let $ G $ be a simple graph with doubly connected blocks $ G_1, G_2, \ldots, G_k $.
Then we have
\[L(\A_G) = L(\A_{G_1}) \times L(\A_{G_2}) \times \cdots \times L(\A_{G_k}).
\]
For each $ 1 \leq i \leq k $,
let $\pi^i$ denote the nice partition of $\A_{G_i}$ induced by $\pi$ naturally.
The previous discussion implies the existence of a maximal modular chain
$\mathcal{C}^i$ corresponding to $\pi^i$.
Lemma \ref{lem-4} ensures that the maximal modular chain of $ L(\A_G) $ induced by $\mathcal{C}^1, \mathcal{C}^2, \ldots, \mathcal{C}^k$ is indeed modular, and it precisely corresponds to the nice partition $\pi$.
\end{proof}

	\begin{exam}\label{ex:mc}
Figure \ref{mc} (a) illustrates the chordal graph $G$ along with a nice partition of $\A_G$.
Let $ \pi = \{\pi_1,\pi_2,\pi_3,\pi_4 \} $ denote the nice partition of $\A_G$, where $\pi_1 = \{H_{34}\}$, $\pi_2 = \{H_{35}, H_{45}\}$, $\pi_3 = \{H_{13}, H_{14}, H_{15}\}$, and $\pi_4 = \{H_{12}, H_{23}, H_{25}\}$.
In the figure, these parts are represented by green, yellow, blue, and red, respectively, with their common vertices being $4$, $5$, $1$, and $2$. The directed version $ D(G) $, shown in Figure \ref{mc} (b), contains no directed cycles, thus providing a perfect elimination order: $2 \prec 1 \prec 5 \prec 4 \prec 3$.
Now, take $X_1=H_{34}$, $X_2=X_1\cap H_{45}$, $X_3=X_2\cap H_{15}$, and $X_4=X_3 \cap H_{12}$, yielding a maximal modular chain $V < X_1 < X_2 < X_3 < X_4$ of $L(\A_G)$.
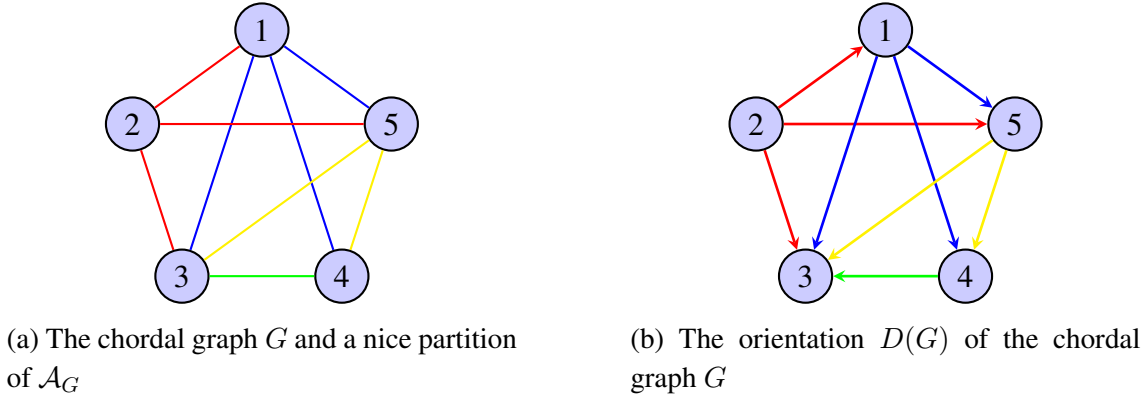
\begin{figure}[ht]
  \centering
  \begin{subfigure}[b]{0.45\textwidth}
    \centering
    \begin{tikzpicture}[thick, scale=0.9]
      \tikzstyle{vertex} = [circle, draw, fill=blue!20, minimum size=20pt, inner sep=0pt]
      \node[vertex] (1) at (90:2) {1};
      \node[vertex] (2) at (162:2) {2};
      \node[vertex] (3) at (234:2) {3};
      \node[vertex] (4) at (306:2) {4};
      \node[vertex] (5) at (18:2) {5};
      \draw[red] (1) -- (2);
      \draw[blue] (1) -- (3);
      \draw[blue] (1) -- (4);
      \draw[blue] (1) -- (5);
      \draw[red] (2) -- (3);
      \draw[red] (2) -- (5);
      \draw[yellow] (3) -- (5);
      \draw[green] (3) -- (4);
      \draw[yellow] (4) -- (5);
    \end{tikzpicture}
    \caption{The chordal graph $G$ and a nice partition of $\A_G$}
  \end{subfigure}
  \hfill
  \begin{subfigure}[b]{0.45\textwidth}
    \centering
    \begin{tikzpicture}[thick, scale=0.9]
      \tikzstyle{vertex} = [circle, draw, fill=blue!20, minimum size=20pt, inner sep=0pt]
      \node[vertex] (1') at (90:2) {1};
      \node[vertex] (2') at (162:2) {2};
      \node[vertex] (3') at (234:2) {3};
      \node[vertex] (4') at (306:2) {4};
      \node[vertex] (5') at (18:2) {5};
      \draw[red, ->, >=stealth, line width=1pt] (2') -- (1');
      \draw[red, ->, >=stealth, line width=1pt] (2') -- (5');
      \draw[red, ->, >=stealth, line width=1pt] (2') -- (3');
      \draw[blue, ->, >=stealth, line width=1pt] (1') -- (3');
      \draw[blue, ->, >=stealth, line width=1pt] (1') -- (4');
      \draw[blue, ->, >=stealth, line width=1pt] (1') -- (5');
      \draw[yellow, ->, >=stealth, line width=1pt] (5') -- (3');
      \draw[yellow, ->, >=stealth, line width=1pt] (5') -- (4');
      \draw[green, ->, >=stealth, line width=1pt] (4') -- (3');
    \end{tikzpicture}
    \caption{The orientation $ D(G) $ of the chordal graph $ G $}
  \end{subfigure}
  \caption{Illustration of Example \ref{ex:mc}}\label{mc}
\end{figure}

For a general example, see Figure \ref{fig:cg},
where $ G $ has two doubly connected components $ G_1 $ and $ G_2 $, arranged from left to right.
The nice partition $ \pi $ of $\A_G$ given in different colors comes from the nice partitions $\pi^1$ and $\pi^2$ of $G_1$ and $G_2$, respectively.
The nice partitions of $G_1$ and $G_2$ are
\begin{align*}
\pi^1&= \{\{H_{34}\},\{H_{35},H_{45}\},\{H_{13},H_{14},H_{15}\},\{H_{12},H_{25},H_{23}\}\},\\[5pt]
\pi^2&=\{\{H_{67}\},\{H_{56},H_{57}\},\{H_{58},H_{78}\}\}
\end{align*}
with $\pi=\pi^1\cup \pi^2$.
Thus, we can construct a maximal modular chain as follows:
\begin{align*}
X_0 &= V,                 & X_1 &= X_0 \cap H_{34},   & X_2 &= X_1 \cap H_{35},   & X_3 &= X_2 \cap H_{14}, \\[5pt]
X_4 &= X_3 \cap H_{23},   & X_5 &= X_4 \cap H_{67},   & X_6 &= X_5 \cap H_{57},   & X_7 &= X_6 \cap H_{58}.
\end{align*}
Hence, we see that
$ X_0 < X_1 <  \cdots < X_7$
is a maximal modular chain of $L(\A_{G})$,
which induces the nice partition $\pi$ of $\A_G$.	
	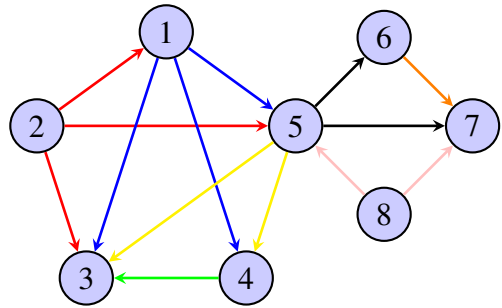
\begin{figure}[H]
	  \centering
	\begin{tikzpicture}[thick, scale=0.9]
  \tikzstyle{vertex} = [circle, draw, fill=blue!20, minimum size=20pt, inner sep=0pt]
  \node[vertex] (1') at (90:2) {1};
  \node[vertex] (2') at (162:2) {2};
  \node[vertex] (3') at (234:2) {3};
  \node[vertex] (4') at (306:2) {4};
  \node[vertex] (5') at (18:2) {5};

  \node[vertex] (7') at ($(5') + (2.6, 0)$) {7};
  \node[vertex] (8') at ($(5') + (1.3, -1.3)$) {8};
  \node[vertex] (6') at ($(5') + (1.3, 1.3)$) {6};

  \draw[red, ->, >=stealth, line width=1pt] (2') -- (1');
  \draw[red, ->, >=stealth, line width=1pt] (2') -- (5');
  \draw[red, ->, >=stealth, line width=1pt] (2') -- (3');
  \draw[blue, ->, >=stealth, line width=1pt] (1') -- (3');
  \draw[blue, ->, >=stealth, line width=1pt] (1') -- (4');
  \draw[blue, ->, >=stealth, line width=1pt] (1') -- (5');
  \draw[yellow, ->, >=stealth, line width=1pt] (5') -- (3');
  \draw[yellow, ->, >=stealth, line width=1pt] (5') -- (4');
  \draw[green, ->, >=stealth, line width=1pt] (4') -- (3');

  \draw[black, ->, >=stealth, line width=1pt] (5') -- (6');
  \draw[black, ->, >=stealth, line width=1pt] (5') -- (7');
  \draw[pink, ->, >=stealth, line width=1pt] (8') -- (5');
  \draw[pink, ->, >=stealth, line width=1pt] (8') -- (7');
  \draw[orange, ->, >=stealth, line width=1pt] (6') -- (7');
\end{tikzpicture}
  \caption{The second illustration of Example \ref{ex:mc}}\label{fig:cg}
	\end{figure}
\end{exam}

\section{The converses of two results by Orlik and Terao}\label{sec-4}

In this section, we will successively prove Theorems \ref{main2} and \ref{main1}, which are the converses of two results by Orlik and Terao.

\begin{proof}[Proof of Theorem \ref{main2}]
	The sufficiency of this conclusion is established in \cite[Corollary 3.88, p.~85]{POHT1992}, noting that $\pi_X$ is a nice partition of $\mathcal{A}_X$ for each $X \in L(\mathcal{A})$.
	We therefore only need to prove the necessity.
	Assume that for any $X \in L(\A)$,
	\begin{equation}\label{eq:chi}
		\chi(\A_X, t) = t^{n-\ell}\prod_{i=1}^{\ell}(t - |\pi_i \cap \A_X|).
	\end{equation}
	Since $\A_X$ is central, we have $\chi(\A_X, 1) = 0$.
	Therefore, there exists some $i \in \{1, \ldots, \ell\}$ such that $|\pi_i \cap \A_X| = 1$.
	On the other hand, for any $p$-section $S$ of $\pi$,
	by taking $X = \cap S$ in \eqref{eq:chi},
	we obtain that
	\[
	\chi(\A_{\cap S}, t) = t^{n-\ell}\prod_{i=1}^{\ell}(t - |\pi_i \cap \A_{\cap S}|).
	\]
	This implies that
	\[
	r(\cap S) = |\{i \mid \pi_i \cap \A_{\cap S} \neq \emptyset\}| \geq |S|,
	\]
	where the equality follows from a basic fact on characteristic polynomials,
	see \cite[Exercise (6), page~401]{PRS2007};
	and naturally, $r(\cap S) \leq |S|$ holds.
	Combining these, we get $r(\cap S) = |S|$, which means $S$ is independent. Consequently, $\pi$ is independent, and thus nice.
\end{proof}

To prove Theorem \ref{main1}, we first provide an equivalent characterization of modular elements based on \cite[Theorem 3.2]{TB1975} by Brylawski.

\begin{lemma}\label{2.1}
	An element $ X \in L(\mathcal{A}) $ is modular if and only if $ \mathcal{A}_X \cap \mathcal{A}_Y \neq \emptyset $ for any $ Y \in L(\mathcal{A})$ with $r(Y) = r - r(X) + 1$.
\end{lemma}

\begin{proof}
	Let $ X \in L(\mathcal{A}) $ be a modular element.
	Assume for contradiction that there exists $ Y \in L(\mathcal{A}) $ with $ r(Y) = r - r(X) + 1 $ such that $ \mathcal{A}_X \cap \mathcal{A}_Y = \emptyset $.
	Combining this with the definition of modular element,
	we see that
	\[
	r + 1 = r(X) + r(Y) = r(X \wedge Y) + r(X \vee Y) = 0 + r(X \vee Y),
	\]
	which contradicts the fact that $ r(\mathcal{A}) = r $.
	
	Conversely, \cite[Theorem 3.2]{TB1975} states that $ X \in L(\mathcal{A}) $ is modular if and only if all complements of $ X $ in $ L(\mathcal{A}) $ are incomparable.
	Hence, it suffices to show that under the condition $ \mathcal{A}_X \cap \mathcal{A}_Y \neq \emptyset $ for any $ Y \in L(\mathcal{A}) $ with $r(Y) = r - r(X) + 1$,
	all complements of $X$ in $L(\A)$ have the same rank $ r-r(X) $, and are therefore naturally incomparable.
	Let $ Z $ be any complement of $ X $, i.e., $ Z $ satisfies $ X \wedge Z = V $ and $ X \vee Z = T $.
	On the one hand, it is obvious from our condition that if $ X \wedge Z = V $, then $ r(Z) \leq r - r(X) $. On the other hand,
	 since $ X \vee Z = T$, the semimodularity of $L(\mathcal{A})$ yields
	\[
	r(Z) + r(X) \geq r(X \vee Z) + r(X \wedge Z) = r,
	\]
	which implies $r(Z) \geq r - r(X)$, as desired. This completes the proof.
\end{proof}

\begin{proof}[Proof of Theorem \ref{main1}]
	For any element $X_k$ from  $\mathcal{C}$, we have $r(X_k) = k$.
	By Lemma \ref{2.1},
	we proceed to show that $\mathcal{A}_{X_k} \cap \mathcal{A}_Y \neq \emptyset$ for any $Y \in L(\mathcal{A})$ with $r(Y) = r - r(X) + 1$. Indeed, for such $Y$, we have
	$
	|\pi_Y| = r(Y) = r-k+1.
	$
	An application of the {Pigeonhole Principle} implies that there exists at least one $i \leq k$ such that $\pi_i \cap \mathcal{A}_Y \neq \emptyset$. Consequently, $\mathcal{A}_{X_{k}} \cap \mathcal{A}_Y \neq \emptyset$, which completes the proof.
\end{proof}


Note that a geometric lattice version of Theorem~\ref{main1} follows from Hallam and Sagan's quotient poset approach~\cite[Theorem~18 and Proposition~22]{JB2015}. Our proof is more direct and avoids the use of the general machinery.

\section*{Acknowledgments}
This work was done under the auspices of the National Science Foundation of China (12101613), and by Guangdong Basic and Applied Basic Research Foundation (2025A15\\15010457).

\end{document}